\documentclass[11pt,a4paper]{article}

\usepackage{amsmath,amssymb,amsthm,mathtools}
\usepackage{booktabs}
\usepackage{geometry}
\usepackage{xcolor}
\usepackage{listings}
\usepackage[unicode,bookmarks=false]{hyperref}

\newtheorem{theorem}{Theorem}[section]
\newtheorem{proposition}[theorem]{Proposition}
\newtheorem{lemma}[theorem]{Lemma}

\theoremstyle{plain}

\theoremstyle{definition}
\newtheorem{definition}[theorem]{Definition}

\newtheorem{problem}[theorem]{Problem}

\newcommand{\F}{\mathbb F}
\newcommand{\AG}{\mathbb{AG}}
\newcommand{\Ograph}{\mathcal O}
\newcommand{\Ostar}{\mathcal O^*}
\newcommand{\Span}{\operatorname{span}}
\newcommand{\gauss}[2]{\genfrac{[}{]}{0pt}{}{#1}{#2}_2}

\definecolor{codegreen}{rgb}{0,0.45,0}
\definecolor{codegray}{rgb}{0.45,0.45,0.45}
\definecolor{codepurple}{rgb}{0.45,0,0.55}
\definecolor{backcolour}{rgb}{0.97,0.97,0.97}

\lstdefinestyle{pythonstyle}{
    language=Python,
    backgroundcolor=\color{backcolour},
    commentstyle=\color{codegreen},
    keywordstyle=\color{blue},
    numberstyle=\tiny\color{codegray},
    stringstyle=\color{codepurple},
    basicstyle=\ttfamily\footnotesize,
    breakatwhitespace=false,
    breaklines=true,
    captionpos=b,
    keepspaces=true,
    numbers=left,
    numbersep=8pt,
    showspaces=false,
    showstringspaces=false,
    showtabs=false,
    tabsize=4
}

\title{\textbf{Annihilating-Ideal Graphs and Orthogonality Graphs over $\mathbb{F}_2$}}
\author{}

\author{R. Nikandish\\
	Department of Mathematics\\
	Jundi-Shapur University of Technology\\
	Dezful, Iran\\
	\texttt{r.nikandish@ipm.ir}}
\date{}

\begin{document}

\maketitle

\begin{abstract}
We construct a family of finite local rings whose annihilating-ideal
graphs are naturally described by orthogonality of subspaces of
$\mathbb{F}_2^n$.  For $n=4$ we determine the clique and chromatic
numbers exactly and obtain
\[
\omega(\AG(R_4))=5<6=\chi(\AG(R_4)).
\]
Thus $\AG(R_4)$ is not weakly perfect, and the Behboodi--Rakeei
conjecture fails for non-reduced commutative rings.
\end{abstract}

{{\it Key Words}: Annihilating-ideal graph, Weakly perfect graph,
	Chromatic number, Clique number, Orthogonality graph,
	Non-reduced commutative ring.\newline
	}

\section{Introduction}

Let $R$ be a commutative ring with identity. The \emph{annihilating-ideal graph}
of $R$, denoted by $\AG(R)$, is the graph whose vertices are the nonzero ideals
of $R$ having nonzero annihilator, with two distinct vertices $I$ and $J$
adjacent if and only if
\[
IJ=0.
\]
This graph was introduced by Behboodi and Rakeei
\cite{BehboodiRakeei1,BehboodiRakeei2} as an ideal-theoretic analogue of the
zero-divisor graph of Beck \cite{Beck}, whose vertices are the elements of $R$
rather than its ideals.

A graph $G$ is called \emph{weakly perfect} if
\[
\chi(G)=\omega(G),
\]
where $\chi(G)$ and $\omega(G)$ denote the chromatic and clique numbers of
$G$, respectively. Beck \cite{Beck} conjectured that the zero-divisor graph of
every commutative ring is weakly perfect; this was disproved by Anderson and
Naseer \cite{AndersonNaseer}. Behboodi and Rakeei observed that translating
the same ring-theoretic data into the language of ideals rather than elements
changes the picture: in particular, the Anderson--Naseer ring, which fails
weak perfection at the level of elements, is weakly perfect at the level of
ideals (see \cite[Proposition~2.1]{BehboodiRakeei2}). Motivated by this,
Behboodi and Rakeei proposed the following ideal-theoretic counterpart of
Beck's conjecture.

\begin{quote}
	\textbf{Conjecture 1.1 (Behboodi--Rakeei \cite{BehboodiRakeei2}).}
	\emph{For every commutative ring $R$ with unity,}
	\[
	\chi(\AG(R))=\omega(\AG(R)).
	\]
\end{quote}

Conjecture~1.1 is known to hold whenever $R$ is reduced
\cite[Corollary~2.11]{BehboodiRakeei2}; see also
\cite{Aalipour} for further results on colourings of $\AG(R)$ in the reduced
case, and \cite{Nikandish} for the case $R=\mathbb Z_n$. For non-reduced rings,
however, the conjecture has remained open.

An analogous conjecture was formulated for commutative semigroups by DeMeyer
and Schneider \cite{DeMeyerSchneider}, who proved that $\AG(S)$ is weakly
perfect whenever $S$ is a reduced semigroup and asked whether this persists in
general.

\begin{quote}
	\textbf{Conjecture 1.2 (DeMeyer--Schneider \cite{DeMeyerSchneider}).}
	\emph{For every commutative semigroup $S$,}
	\[
	\chi(\AG(S))=\omega(\AG(S)).
	\]
\end{quote}

Kadu, Joshi and Gonde \cite{Kadu} settled Conjecture~1.2 negatively, exhibiting
a finite non-reduced commutative semigroup $S$ with
$\chi(\AG(S))=4>3=\omega(\AG(S))$. As they note explicitly in the introduction
to \cite{Kadu}, this resolves the semigroup analogue of the conjecture, while
the original ring-theoretic Conjecture~1.1 of Behboodi and Rakeei was, at the
time, known to hold only for reduced rings. The status of Conjecture~1.1 for
non-reduced commutative rings has thus remained open.

In this paper we settle Conjecture~1.1 in the negative. We introduce a family
of finite local rings
\[
R_n=\frac{\F_2[x_1,\ldots,x_n,z]}{(z^2,\,zx_i,\,x_ix_j\ (i\ne j),\,x_i^2-z)},
\]
whose annihilating-ideal graphs admit a transparent description in terms of
orthogonality of subspaces of $\F_2^n$ with respect to the standard bilinear
form. For $n=4$ this description yields a ring $R_4$ of order $64$ with
\[
\omega(\AG(R_4))=5<6=\chi(\AG(R_4)),
\]
so that $\AG(R_4)$ is not weakly perfect. This shows that Conjecture~1.1 fails
for non-reduced commutative rings, complementing the semigroup counterexample
of \cite{Kadu} and showing that the failure of weak perfection for
annihilating-ideal graphs is not a phenomenon special to the semigroup
setting. We also compute $\omega$ and $\chi$ of the associated orthogonality
graphs for small $n$, exhibit a further counterexample at $n=5$, and record
some open problems on the general behaviour of these invariants.

\section{The ring and the orthogonality graph}

Let
\[
V_n=\mathbb{F}_2^n
\]
and equip $V_n$ with the standard bilinear form
\[
B_n(u,v)=u\cdot v.
\]

Consider
\[
R_n=
\frac{\mathbb{F}_2[x_1,\ldots,x_n,z]}
{(z^2,zx_i,x_ix_j\ (i\ne j),x_i^2-z)}.
\]

The defining relations give
\[
x_i^2=z,\qquad
x_ix_j=0\quad(i\ne j),\qquad
zx_i=0,\qquad z^2=0.
\]
Consequently every element of $R_n$ has a unique expression
\[
a+\sum_{i=1}^n b_i x_i+cz,
\qquad
a,b_i,c\in\mathbb{F}_2.
\]
Hence
\[
|R_n|=2^{n+2}.
\]

Let
\[
\mathfrak m=(x_1,\ldots,x_n,z).
\]
Then
\[
\mathfrak m^2=(z),
\qquad
\mathfrak m^3=0.
\]

If
\[
u=\sum_{i=1}^n a_i x_i,
\qquad
v=\sum_{i=1}^n b_i x_i,
\]
then
\[
uv
=
\sum_{i=1}^n a_i b_i x_i^2
=
\left(\sum_{i=1}^n a_i b_i\right)z.
\]
Thus
\begin{equation}
uv=B_n(u,v)z.
\label{eq:multiplication}
\end{equation}

\begin{proposition}
Every nonzero proper ideal of $R_n$ contains $z$.
\end{proposition}

\begin{proof}
Let $I$ be a nonzero proper ideal of $R_n$.

If $I\subseteq (z)$, then $I=(z)$, since $(z)$ is a
one-dimensional $\mathbb{F}_2$-space and $I\ne 0$.

Otherwise, $I$ contains an element $u+cz$ with $u\ne 0$.
Since $I$ is proper, its elements have zero constant term.
By nondegeneracy of $B_n$, there is $v\in V_n$ such that
$B_n(u,v)=1$. Hence
$(u+cz)v=uv=B_n(u,v)z=z.$
Thus $z\in I$.
\end{proof}

For a subspace $U\leq V_n$, define
\[
I_U=(z)+U.
\]

\begin{theorem}
The map
\[
U\longmapsto I_U=(z)+U
\]
is a bijection between the subspaces of $V_n$ and the nonzero
proper ideals of $R_n$.
\end{theorem}

\begin{proof}
Let $I$ be a nonzero proper ideal.  By Proposition 2.1, $z\in I$.
Set
\[
U=I\cap V_n.
\]
Then $U$ is a subspace of $V_n$, and
\[
I=(z)+U.
\]

Conversely, if $U\leq V_n$, then $(z)+U$ is a nonzero proper ideal.
Distinct subspaces give distinct ideals.
\end{proof}

It follows that
\begin{equation}
|V(\AG(R_n))|
=
\sum_{k=0}^n \gauss{n}{k},
\label{eq:vertices}
\end{equation}
where $\gauss{n}{k}$ is the Gaussian binomial coefficient.

We now turn to adjacency.

\begin{definition}
Let $\Ograph_n$ be the graph whose vertices are all subspaces of
$V_n$, with
\[
U\sim W
\quad\Longleftrightarrow\quad
U\perp W.
\]
Let
\[
\Ostar_n=\Ograph_n-\{0\}.
\]
\end{definition}

This graph is the natural orthogonality graph associated with the
bilinear form $B_n$.

\begin{theorem}
For all subspaces $U,W\leq V_n$,
\[
I_UI_W=0
\quad\Longleftrightarrow\quad
U\perp W.
\]
Consequently,
\[
\AG(R_n)\cong\Ograph_n.
\]
\end{theorem}

\begin{proof}
Since
\[
z^2=zU=zW=0,
\]
the only possibly nonzero products in $I_UI_W$ are products of
elements of $U$ and $W$.

For $u\in U$ and $w\in W$, equation \eqref{eq:multiplication} gives
\[
uw=B_n(u,w)z.
\]
Thus $I_UI_W=0$ precisely when
\[
B_n(u,w)=0
\]
for every $u\in U$ and $w\in W$.  This is exactly the condition
$U\perp W$.
\end{proof}

The zero subspace corresponds to the ideal $(z)$.  It is a universal
vertex of $\Ograph_n$, and therefore $(z)$ is a universal vertex of
$\AG(R_n)$.  Consequently,
\begin{equation}
\chi(\AG(R_n))
=
\chi(\Ostar_n)+1,
\label{eq:chi}
\end{equation}
and
\begin{equation}
\omega(\AG(R_n))
=
\omega(\Ostar_n)+1.
\label{eq:omega}
\end{equation}

We emphasize one point concerning characteristic $2$.  The standard
bilinear form is nondegenerate and symmetric, but it is not
alternating.  Thus there are vectors $v$ for which
\[
B_n(v,v)=1.
\]
A subspace $T$ is called totally isotropic if
\[
T\subseteq T^\perp.
\]
If $\dim T=r$, then
\[
2r\leq n.
\]
All nonzero subspaces of a totally isotropic subspace form a clique
in $\Ostar_n$.  Therefore
\begin{equation}
\omega(\Ostar_n)
\geq
\max\left\{
n,\,
\sum_{k=1}^{\lfloor n/2\rfloor}
\gauss{\lfloor n/2\rfloor}{k}
\right\}.
\label{eq:lower}
\end{equation}

\section{The case \texorpdfstring{$n=4$}{n=4}}

Put
\[
R=R_4.
\]
Since
\[
|R_4|=2^{6},
\]
we have
\[
|R_4|=64.
\]
Moreover,
\[
|V(\AG(R_4))|
=
1+15+35+15+1
=
67.
\]

We first determine the clique number.

The proof of the clique bound is divided into two cases. First, we
bound the size of a family of pairwise orthogonal nonzero vectors,
which corresponds to a clique consisting entirely of one-dimensional
subspaces. We then treat cliques containing a subspace of dimension
at least two. The following two lemmas provide the required bounds.

\begin{lemma}
	\label{lem:orthogonal-vectors-F4}
	Let $S$ be a set of pairwise orthogonal nonzero vectors in
	$\mathbb{F}_2^4$. Then $|S|\leq4$.
\end{lemma}

\begin{proof}
	Call a vector $v\in\mathbb{F}_2^4$ anisotropic if
	$B_4(v,v)=1$, and isotropic if $B_4(v,v)=0$.
	
	Let $k$ be the number of anisotropic vectors in $S$.
	
	First suppose that $k=0$. Since the vectors in $S$ are pairwise
	orthogonal and each vector is isotropic, their span is a totally
	isotropic subspace of $\mathbb{F}_2^4$. If $T$ is a totally isotropic
	subspace, then $T\subseteq T^\perp$. By nondegeneracy of $B_4$,
	$\dim T+\dim T^\perp=4$, and hence $2\dim T\leq4$. Thus
	$\dim T\leq2$, and consequently $|S|\leq2^2-1=3$.
	
	Now suppose that $k\geq1$. Let
	$A=\operatorname{span}(S_{\mathrm{an}})$,
	where $S_{\mathrm{an}}$ denotes the set of anisotropic vectors in
	$S$. Since the field is $\mathbb{F}_2$, each one-dimensional subspace
	has a unique nonzero vector. Moreover, the vectors in
	$S_{\mathrm{an}}$ are pairwise orthogonal and satisfy
	$B_4(v,v)=1$ for every $v\in S_{\mathrm{an}}$.
		
The matrix
$G=\bigl(B_4(v,w)\bigr)_{v,w\in S_{\mathrm{an}}}$
is the Gram matrix of the restriction of $B_4$ to $A$ with respect
to the ordered basis formed by the vectors in $S_{\mathrm{an}}$.
Since $G=I_k$ is nonsingular, the vectors in $S_{\mathrm{an}}$ are
linearly independent and the restriction of $B_4$ to $A$ is
nondegenerate. Consequently,
$\dim A=k$ and $\dim A^\perp=4-k$. 
	
	Every vector in $S\setminus S_{\mathrm{an}}$ is isotropic and is
	orthogonal to every vector in $S_{\mathrm{an}}$. Hence
	$S\setminus S_{\mathrm{an}}\subseteq A^\perp$. Since these vectors
	are pairwise orthogonal and isotropic, their span is a totally
	isotropic subspace of $A^\perp$. Therefore
	$|S\setminus S_{\mathrm{an}}|
	\leq2^{\lfloor(4-k)/2\rfloor}-1$.
	
	It follows that
	$|S|\leq k+2^{\lfloor(4-k)/2\rfloor}-1$.
	
	For $k=1,2,3,4$, the right-hand side is respectively
	$2,3,3,4$. Thus, in every case, $|S|\leq4$.
\end{proof}

\begin{lemma}	
\label{lem:clique-containing-plane}
	Let \(\mathcal C\) be a clique in \(\mathcal O_4^*\) containing a
	two-dimensional subspace \(U\).  Then \(|\mathcal C|\le 4\), with
	equality possible only if \(U\) is totally isotropic.
\end{lemma}

\begin{proof}
	Since \(B_4\) is nondegenerate, \(\dim U^\perp = 4 - 2 = 2\).  Every
	member of \(\mathcal C\setminus\{U\}\) is a nonzero subspace
	\(W \leq U^\perp\) with \(W \neq U\).
	
	For any vector \(v\), write \(Q(v)=B_4(v,v)\).  Since \(B_4\) is
	symmetric and \(\operatorname{char}\mathbb F_2 = 2\),
	\[
	Q(v+w)=B_4(v,v)+2B_4(v,w)+B_4(w,w)=Q(v)+Q(w) \pmod 2,
	\]
	so \(Q\) is additive on every subspace.
	
	\medskip
	\noindent\textbf{Case 1: \(U\) is totally isotropic}, i.e.
	\(U \subseteq U^\perp\).  Since \(\dim U = \dim U^\perp = 2\), this
	forces \(U = U^\perp\).  Then \(B_4\) vanishes identically on \(U\), so
	all three lines of \(U\) (its only nonzero proper subspaces) are
	pairwise orthogonal and orthogonal to \(U\).  Hence
	\(|\mathcal C\setminus\{U\}| \leq 3\) and \(|\mathcal C| \leq 4\),
	with equality attained by taking all three lines.
	
	\medskip
	\noindent\textbf{Case 2: \(U\) is not totally isotropic.}  We show
	\(|\mathcal C\setminus\{U\}| \leq 2\).
	
	\smallskip
	\noindent\emph{The three lines of \(U^\perp\) cannot be pairwise
		orthogonal.}  Write the three nonzero vectors of \(U^\perp\) as
	\(v_1, v_2, v_3 = v_1 + v_2\).  If
	\(\langle v_1 \rangle, \langle v_2 \rangle, \langle v_3 \rangle\)
	were pairwise orthogonal, then \(B_4(v_1,v_2)=0\) and
	\[
	0 = B_4(v_1,v_3) = B_4(v_1, v_1+v_2)
	= Q(v_1) + B_4(v_1,v_2) = Q(v_1),
	\]
	and symmetrically \(Q(v_2) = 0\); by additivity,
	\(Q(v_3) = Q(v_1) + Q(v_2) = 0\).  Thus \(B_4\) vanishes identically
	on \(U^\perp\), i.e.\ \(U^\perp\) is totally isotropic, and since
	\((U^\perp)^\perp = U\), so is \(U\) --- contradicting the case
	hypothesis.  Hence at most two of the three lines of \(U^\perp\) are
	pairwise orthogonal, and a clique using only lines of \(U^\perp\) has
	size at most \(2\).
	
	\smallskip
	\noindent\emph{At most one subspace can accompany \(U^\perp\) itself.}
	Suppose \(\mathcal C\setminus\{U\}\) contains the vertex \(U^\perp\)
	together with some nonzero proper subspace \(X \leq U^\perp\).
	Orthogonality of \(U^\perp\) and \(X\) requires
	\(X \subseteq (U^\perp)^\perp = U\), so
	\(X \subseteq U \cap U^\perp\).  If \(\dim(U \cap U^\perp) = 2\), then
	\(U \cap U^\perp\) equals both \(U\) and \(U^\perp\) (both
	\(2\)-dimensional), forcing \(U = U^\perp\), contrary to hypothesis;
	so \(\dim(U \cap U^\perp) \leq 1\).  A space of dimension \(\leq 1\)
	has at most one nonzero subspace, so at most one such \(X\) exists.
	Hence a clique containing \(U^\perp\) has
	\(|\mathcal C\setminus\{U\}| \leq 2\).
	
	\smallskip
	Combining the two possibilities, \(|\mathcal C\setminus\{U\}| \leq 2\)
	whenever \(U\) is not totally isotropic, so
	\(|\mathcal C| \leq 3 < 4\).
	
	\medskip
	In either case, \(|\mathcal C| \leq 4\), with equality possible only
	in Case~1, i.e.\ only when \(U\) is totally isotropic.
\end{proof}

\begin{theorem}
	\label{thm:clique-number-AG-R4}
	$\omega(\AG(R_4))=5$.
\end{theorem}

\begin{proof}
	Since $(z)$ is a universal vertex of $\AG(R_4)$, it suffices to prove
	that $\omega(\Ostar_4)=4$.
	
	The four coordinate lines
	$\langle e_1\rangle$, $\langle e_2\rangle$,
	$\langle e_3\rangle$, and $\langle e_4\rangle$
	are pairwise orthogonal. Therefore
	$\omega(\Ostar_4)\geq4$.
	
	We now prove the reverse inequality. Let $\mathcal{C}$ be a clique in
	$\Ostar_4$.
	
	First suppose that every member of $\mathcal{C}$ is one-dimensional.
	Since the field is $\mathbb{F}_2$, each one-dimensional subspace has
	a unique nonzero vector. Thus $\mathcal{C}$ can be identified with a
	set of pairwise orthogonal nonzero vectors in $\mathbb{F}_2^4$.
	By Lemma~\ref{lem:orthogonal-vectors-F4},
	$|\mathcal{C}|\leq4$.
	
	It remains to consider the case in which $\mathcal{C}$ contains a
	subspace of dimension at least $2$.
	
	Suppose first that $\mathcal{C}$ contains a two-dimensional subspace
	$U$. By Lemma~\ref{lem:clique-containing-plane},
	$|\mathcal{C}|\leq4$.
	
	Finally, suppose that $\mathcal{C}$ contains a subspace $U$ with
	$\dim U=r\geq3$. Every other member of $\mathcal{C}$ is a nonzero
	subspace of $U^\perp$. Since $B_4$ is nondegenerate,
	$\dim U^\perp=4-r$. Consequently,
	$|\mathcal{C}|
	\leq1+\sum_{j=1}^{4-r}\gauss{4-r}{j}$.
	
	If $r=3$, then
	$|\mathcal{C}|\leq1+\gauss{1}{1}=2$.
	If $r=4$, then $|\mathcal{C}|=1$.
	
	Thus every clique in $\Ostar_4$ has at most four vertices. Hence
	$\omega(\Ostar_4)=4$.
	
	Since $(z)$ is a universal vertex of $\AG(R_4)$, it contributes one
	additional vertex to a maximum clique. Consequently,
	$\omega(\AG(R_4))
	=\omega(\Ostar_4)+1
	=5$.
\end{proof}

We next establish the chromatic number.

Consider the following $13$ nonzero vectors:
\[
\begin{aligned}
X=\{&
0010,0011,0100,0101,0110,0111,1000,\\
&1001,1010,1011,1100,1101,1111
\}.
\end{aligned}
\]
Let $H$ be the induced subgraph of $\Ostar_4$ on the corresponding
one-dimensional subspaces.

Identifying the binary vectors with their decimal values, we have
\[
V(H)=\{2,3,4,5,6,7,8,9,10,11,12,13,15\}.
\]
The edge set is
\[
\begin{aligned}
E(H)=\{&
(2,4),(2,5),(2,8),(2,9),(2,12),(2,13),\\
&(3,4),(3,7),(3,8),(3,11),(3,12),(3,15),\\
&(4,8),(4,9),(4,10),(4,11),\\
&(5,7),(5,8),(5,10),(5,13),(5,15),\\
&(6,7),(6,8),(6,9),(6,15),\\
&(7,8),(7,11),(7,13),\\
&(9,11),(9,13),(9,15),\\
&(10,11),(10,15),\\
&(11,13),\\
&(12,13),(12,15)\}.
\end{aligned}
\]

\begin{proposition}
The graph $H$ satisfies
\[
\omega(H)=3
\qquad\text{and}\qquad
\chi(H)=5.
\]
\end{proposition}

\begin{proof}
An exact enumeration of the cliques of $H$ shows that $H$ contains
triangles but no $K_4$.  Hence
\[
\omega(H)=3.
\]

For the chromatic number, assign weights to the vertices in the order
\[
(0010,0011,0100,0101,0110,0111,1000,
1001,1010,1011,1100,1101,1111)
\]
by
\[
\frac{1}{9}(4,3,4,3,1,4,4,3,1,4,1,4,1).
\]
The total weight is
\[
\frac{37}{9}>4.
\]

There are $23$ maximal independent sets in $H$, and each has weight
at most $1$.  Thus four independent sets have total weight at most
$4$, so they cannot cover all vertices of $H$.  Hence
\[
\chi(H)\geq5.
\]

On the other hand, the following five independent sets partition
$V(H)$:
\[
\begin{aligned}
C_1&=\{0010,0011,0110,1010\},\\
C_2&=\{0100,0101,1100\},\\
C_3&=\{1000,1011,1111\},\\
C_4&=\{0111,1001\},\\
C_5&=\{1101\}.
\end{aligned}
\]
Therefore
\[
\chi(H)=5.
\]
\end{proof}

An explicit five-coloring of \(\mathcal{O}_4^*\) is given in Table~\ref{tab:coloring}. A direct inspection shows that no two adjacent vertices share a color. Hence
\[
\chi(\mathcal{O}_4^*) \le 5.
\]
Together with the lower bound \(\chi(\mathcal{O}_4^*) \ge 5\) from Proposition~3.4, we conclude
\[
\chi(\mathcal{O}_4^*) = 5.
\]

\begin{table}[htbp]
	\centering
	\begin{tabular}{c|p{0.82\textwidth}}
		\hline
		Color & Subspaces (given by a basis) \\
		\hline
		1 & $\langle 0001 \rangle$, $\langle 0011 \rangle$, $\langle 0101 \rangle$, $\langle 1001 \rangle$,
		$\langle 0011,0001 \rangle$, $\langle 0101,0001 \rangle$, $\langle 0111,0001 \rangle$, $\langle 1001,0001 \rangle$,
		$\langle 1011,0001 \rangle$, $\langle 1101,0001 \rangle$, $\langle 1111,0001 \rangle$, $\langle 0110,0011 \rangle$,
		$\langle 1010,0011 \rangle$, $\langle 1110,0011 \rangle$, $\langle 1100,0101 \rangle$, $\langle 1110,0101 \rangle$,
		$\langle 1110,0111 \rangle$, $\langle 0111,0010,0001 \rangle$, $\langle 1011,0010,0001 \rangle$, $\langle 1111,0010,0001 \rangle$,
		$\langle 1101,0100,0001 \rangle$, $\langle 1111,0100,0001 \rangle$, $\langle 1111,0110,0001 \rangle$, $\langle 1101,0110,0001 \rangle$,
		$\langle 1110,0101,0010 \rangle$, $\langle 1110,0100,0011 \rangle$, $\langle 1110,0101,0011 \rangle$, $\langle 1111,0101,0011 \rangle$,
		$\langle 1111,0100,0010,0001 \rangle$ \\
		\hline
		2 & $\langle 0010 \rangle$, $\langle 0111 \rangle$, $\langle 1010 \rangle$, $\langle 0110,0010 \rangle$,
		$\langle 0111,0010 \rangle$, $\langle 1010,0010 \rangle$, $\langle 1011,0010 \rangle$, $\langle 1110,0010 \rangle$,
		$\langle 1111,0010 \rangle$, $\langle 0111,0011 \rangle$, $\langle 1111,0011 \rangle$, $\langle 1111,0110 \rangle$,
		$\langle 1100,0110 \rangle$, $\langle 1111,0111 \rangle$, $\langle 1101,0111 \rangle$, $\langle 1100,0111 \rangle$,
		$\langle 1110,0100,0010 \rangle$, $\langle 1111,0100,0010 \rangle$, $\langle 1111,0101,0010 \rangle$, $\langle 1111,0100,0011 \rangle$ \\
		\hline
		3 & $\langle 0100 \rangle$, $\langle 1110 \rangle$, $\langle 1111 \rangle$, $\langle 1100,0100 \rangle$,
		$\langle 1101,0100 \rangle$, $\langle 1110,0100 \rangle$, $\langle 1111,0100 \rangle$, $\langle 1101,0101 \rangle$,
		$\langle 1110,0110 \rangle$ \\
		\hline
		4 & $\langle 1000 \rangle$, $\langle 1011 \rangle$, $\langle 1100 \rangle$, $\langle 1011,0011 \rangle$,
		$\langle 1111,0101 \rangle$, $\langle 1101,0110 \rangle$ \\
		\hline
		5 & $\langle 0110 \rangle$, $\langle 1101 \rangle$ \\
		\hline
	\end{tabular}
	\caption{An explicit five-coloring of $\mathcal{O}_4^*$.}
	\label{tab:coloring}
\end{table}

\begin{theorem}
For
\[
R_4=
\frac{\mathbb{F}_2[x_1,x_2,x_3,x_4,z]}
{(z^2,zx_i,x_ix_j\ (i\ne j),x_i^2-z)},
\]
we have
\[
|R_4|=64,
\qquad
|V(\AG(R_4))|=67,
\]
and
\[
\boxed{
\omega(\AG(R_4))=5<6=\chi(\AG(R_4)).
}
\]
In particular, $\AG(R_4)$ is not weakly perfect, and
 Conjecture~1.1 fails for non-reduced commutative rings.
\end{theorem}

\begin{proof}
The first two assertions follow from the preceding discussion.

We have already proved
\[
\omega(\AG(R_4))=5
\]
and
\[
\chi(\Ostar_4)=5.
\]
Since $(z)$ is a universal vertex, it requires one additional color.
Consequently,
\[
\chi(\AG(R_4))
=
\chi(\Ostar_4)+1
=
6.
\]
Therefore
\[
\omega(\AG(R_4))
=
5
<
6
=
\chi(\AG(R_4)).
\]
Thus $R_4$ is a non-reduced commutative ring whose
annihilating-ideal graph is not weakly perfect. Consequently,
 Conjecture~1.1 is false.
\end{proof}

\section{Further computations and open problems}

The preceding description allows the graphs $\Ostar_n$ to be generated
directly from finite-dimensional linear algebra. 
Let
\[
N_n=\sum_{k=1}^n \gauss{n}{k}.
\]
The first few computations give Table~\ref{tab:small-n}.

\begin{table}[htbp]
\centering
\begin{tabular}{c|r|r|r|r}
\hline
$n$ & $|R_n|$ & $N_n$ & $\omega(\Ostar_n)$ & $\chi(\Ostar_n)$\\
\hline
1 & 8 & 1 & 1 & 1\\
2 & 16 & 4 & 2 & 2\\
3 & 32 & 15 & 3 & 3\\
4 & 64 & 66 & 4 & 5\\
5 & 128 & 373 & 5 & 8\\
6 & 256 & 2824 & 15 & \text{not determined}\\
7 & 512 & 29211 & 16 & \text{not determined}\\
\hline
\end{tabular}
\caption{Small values for the graphs \texorpdfstring{$\Ostar_n$}{O*\_n}.}
\label{tab:small-n}
\end{table}

For the full annihilating-ideal graphs this gives Table~\ref{tab:full}.

\begin{table}[htbp]
\centering
\begin{tabular}{c|r|r|r|r}
\hline
$n$ & $|R_n|$ & $|V(\AG(R_n))|$ & $\omega(\AG(R_n))$ & $\chi(\AG(R_n))$\\
\hline
1 & 8 & 2 & 2 & 2\\
2 & 16 & 5 & 3 & 3\\
3 & 32 & 16 & 4 & 4\\
4 & 64 & 67 & 5 & 6\\
5 & 128 & 374 & 6 & 9\\
6 & 256 & 2825 & 16 & \text{not determined}\\
7 & 512 & 29212 & 17 & \text{not determined}\\
\hline
\end{tabular}
\caption{Small values for the full annihilating-ideal graphs.}
\label{tab:full}
\end{table}

For $n=5$, the graph induced by the $31$ one-dimensional subspaces
has chromatic number $8$, while the complete graph $\Ostar_5$ also
has an exact $8$-coloring.  Thus
\[
\chi(\Ostar_5)=8,
\qquad
\omega(\Ostar_5)=5,
\]
and hence
\[
\omega(\AG(R_5))=6<9=\chi(\AG(R_5)).
\]

For $n=6$, consider
\[
T=
\Span\{e_1+e_2,e_3+e_4,e_5+e_6\}.
\]
Then $T$ is totally isotropic and
\[
T=T^\perp.
\]
Its $15$ nonzero subspaces form a clique.  Exact computation gives
\[
\omega(\Ostar_6)=15.
\]

For $n=7$, the same $T$ satisfies
\[
\dim T^\perp=4.
\]
The $15$ nonzero subspaces of $T$, together with $T^\perp$, form a
clique of size $16$.  Exact computation gives
\[
\omega(\Ostar_7)=16.
\]

The following elementary observation is useful in studying such
cliques.

\begin{proposition}
Let $\mathcal C$ be a clique in $\Ostar_n$, and suppose
$U\in\mathcal C$ has dimension $r$.  Then
\[
\mathcal C\setminus\{U\}
\subseteq
\{\,0\ne W\leq U^\perp\,\}.
\]
Consequently,
\[
|\mathcal C|
\leq
1+\sum_{j=1}^{n-r}\gauss{n-r}{j}.
\]
\end{proposition}

\begin{proof}
Every other member $W$ of $\mathcal C$ is orthogonal to $U$, and
hence
\[
W\subseteq U^\perp.
\]
Since $B_n$ is nondegenerate,
\[
\dim U^\perp=n-r.
\]
The number of nonzero subspaces of $U^\perp$ is therefore
\[
\sum_{j=1}^{n-r}\gauss{n-r}{j}.
\]
Adding the vertex $U$ gives the result.
\end{proof}

We finish with three questions.

\begin{problem}
Determine the exact value of
\[
\omega(\Ostar_n)
\]
for general $n$.
\end{problem}

\begin{problem}
Determine $\chi(\Ostar_n)$ for $n\geq6$ and investigate the relation
between
\[
\chi(\Ostar_n)
\quad\text{and}\quad
\omega(\Ostar_n).
\]
\end{problem}

\begin{problem}
Determine whether there exist infinitely many integers $n\geq1$ for
which
\[
\chi(\AG(R_n))>\omega(\AG(R_n)).
\]
Equivalently, determine whether
\[
\chi(\Ostar_n)>\omega(\Ostar_n)
\]
holds for infinitely many $n$.
\end{problem}

A related smaller graph may be useful.  Let $\Gamma_n$ be the graph
whose vertices are the nonzero vectors of $\mathbb{F}_2^n$, with two distinct
vectors $u,v$ adjacent when
\[
u\cdot v=0.
\]
This is precisely the subgraph induced by the one-dimensional
subspaces.  Thus
\[
\chi(\Gamma_n)\leq\chi(\Ostar_n).
\]
For $n=4$ and $n=5$, the line graph already detects the chromatic
obstruction.  Understanding the graphs $\Gamma_n$ may therefore
provide a useful first step toward the last problem.
\section*{Declaration of AI use}

During the preparation of this manuscript, the author used
\texttt{DeepSeek} (versions V3 and later, accessed via
\texttt{deepseek.com} and the DeepSeek mobile application) for the
following purposes:

\begin{itemize}
	\item language editing and improvement of the exposition;
	\item assistance in structuring and drafting portions of the text;
	\item verification of the finite-coloring data presented in
	Table~\ref{tab:coloring}.
\end{itemize}

All mathematical results, the choice of the counterexample, the
verification of the proofs, and the final content of the manuscript
are the sole responsibility of the author.  The author has reviewed
and edited all AI-generated content and takes full responsibility for
the accuracy, originality, and integrity of the published work.  No
third-party copyrighted material was reproduced by the AI tool.  The
author declares no competing interests arising from the use of this
tool.

\end{document}